\documentclass[12pt, reqno, oneside]{amsart}
\usepackage{amsmath, amsthm, amssymb, amscd, array}
\usepackage{enumitem, geometry, graphics, graphicx, pslatex}
\usepackage[colorlinks=true, pdfstartview=FitV, linkcolor=black, citecolor=black, urlcolor=black]{hyperref}

\usepackage{diagbox} 

\theoremstyle{plain}
\newtheorem{theorem}{Theorem}

\newtheorem{problem}[theorem]{Problem} 

\theoremstyle{definition}
\newtheorem{definition}[theorem]{Definition}
\newtheorem{remark}[theorem]{Remark}
\newtheorem{example}[theorem]{Example} 

\allowdisplaybreaks

\makeatletter
\@addtoreset{footnote}{page}
\makeatother

\begin{document}

\title{A Snail Race Problem}

\author{Yu Chen and Wenxiang Zhu}
\address{Department of Mathematics and Statistics, Idaho State University \\                  
Department of Computer Science, Mathematics, and Physics, Baker University}

\maketitle

\begin{abstract}
Inspired by Problem~17 from the 2024 American Mathematics Competition (AMC) 10B, this work focuses on enumerating the distinct outcomes of a snail race with specified number of ties of a certain type. 
We begin by developing a recurrence relation and subsequently derive a closed-form formula for the number of possible outcomes using the exponential generating function method.
Two special cases of the problem are considered in detail.
Our analysis also explores the connections between the solution to this problem and the ordered Bell numbers, Stirling numbers of the second kind, and partial Bell polynomials. 
\end{abstract}

\section{Introduction} \label{intro} 

Problem~17 from the 2024 American Mathematics Competition (AMC) 10B states:
``\textit{In a race among 5 snails, there is at most one tie, but that tie can involve any number of snails. 
For example, the result of the race might be that Dazzler is first; Abby, Cyrus, and Elroy are tied for second; and Bruna is fifth. 
How many different results of the race are possible?}"
This intriguing problem in elementary combinatorics naturally leads to our exploration of the following general problem.

\begin{problem}[Snail Race Problem] \label{prob:snail-race}
Find the number $r_m(n,k)$ of possible outcomes in a race among $n$ distinct snails that contains exactly $k$ ties of size at least $m$, where $n$ and $m$ are positive integers and $k$ is a nonnegative integer.
\end{problem}

An outcome of the race is an ordered partition of a set of $n$ elements and a tie of size $m$ is a block in the partition with exactly $m$ elements.
Using the symbol introduced in Problem~\ref{prob:snail-race}, we can express the answer to Problem~17 from the 2024 AMC 10B as $r_2(5,1)$.
This problem is closely related to several well-established combinatorial concepts, e.g., the partitions of a finite set, Stirling numbers of the second kind, ordered Bell numbers, and partial Bell polynomials. 
The Stirling number of the second kind, $S(n,k)$, counts the ways to partition a set of $n$ elements into exactly $k$ nonempty subsets (Graham \textit{et al.} \cite{GKP94}). 
The ordered Bell number, $J_n$, counts the ways to partition a set of $n$ elements into ordered nonempty subsets, which can also be interpreted as the total possible outcomes in a race with $n$ participants or the combinations for an $n$-button combination lock. 
The properties of ordered Bell numbers and their exponential generating function have been extensively investigated (see Bell \cite{BellNum}, Cayley \cite{Cayley1859}, Good \cite{Good75}, Gross \cite{Gross62}, Mendelson \cite{EM82}, and Velleman and Call \cite{VC95}). 

In \cite{EM82}, Mendelson applied the inclusion-exclusion principle to derive a closed-form formula for $J_{nk}$, the number of all possible outcomes in a race of $n$ participants where only the first $k$ places are occupied. Subsequently, $J_n$ is obtained via $J_n=\sum_{k=1}^n J_{nk}$. 
The combinatorial numbers, $r_m(n,k)$, also satisfy the equation $J_n=\sum_{k=0}^{\infty} r_m(n,k)$; moreover, $r_1(n,k)=J_{nk}$.
To the best of our knowledge, Problem~\ref{prob:snail-race} has not been studied in the existing literature.

\section{Recurrence Relation} \label{RR} 

Let $\mathbb{Z}$, $\mathbb{Z}_{\geqslant 0}$, and $\mathbb{N}$ denote the sets of all integers, nonnegative integers, and natural numbers, respectively.
Let $k, n \in \mathbb{Z}$ and $m \in \mathbb{N}$.
For convenience, we define $r_m(0,0)=1$, $r_m(0,k)=0$ for $k \geqslant 1$, and $r_m(n,k)=0$ for either $n<0$ or $k<0$.  

To gain familiarity with Problem~\ref{prob:snail-race}, we begin by considering several specific cases.

\begin{example} 
Let $m, n \in \mathbb{N}$ and $k \in \mathbb{Z}_{\geqslant 0}$. 

(1) When $k=0$ and $m=2$, there is no tie of size greater than $1$, which implies $r_2(n,0)=n!$. 

(2) For $r_m(n,k) \ne 0$, we must have $km \leqslant n$.
Hence $r_m(n,k)=0$ if $k> \left\lfloor \frac{n}{m} \right\rfloor$, where $\lfloor x \rfloor$ represents the greatest integer less than or equal to a real number $x$. 

(3) A race with $mn$ participants that results in exactly $n$ ties of size at least $m$ is equivalent to partitioning a set of $mn$ elements into the $n$ ordered subsets of size $m$. 
Thus, 
\begin{equation*}
 r_m(mn,n)
=\binom{mn}{m} 
  \binom{m(n-1)}{m}
  \binom{m(n-2)}{m} 
  \cdots 
  \binom{m}{m}
=\frac{(mn)!}{(m!)^n}.
\end{equation*}
\end{example}

For the general case, we have the following recurrence relation.

\begin{theorem} \label{thm_rcr1} 
For $m, n \in \mathbb{N}$ and $k \in \mathbb{Z}_{\geqslant 0}$, 
\begin{equation} \label{recur1} 
 r_m(n,k)
=\sum_{i=1}^{m-1} \binom{n}{i} r_m(n-i,k) + \sum_{i=m}^n \binom{n}{i} r_m(n-i,k-1),
\end{equation}
\end{theorem}

\begin{proof} 
For $1 \leqslant i \leqslant n$, there are exactly $\binom{n}{i}$ ways to choose a group of $i$ participants to tie for the first place among $n$ participants, and then consider two cases for the remaining $n-i$ participants.

\textit{Case~1:} 
If $i \leqslant m-1$, then there are exactly $k$ ties of size at least $m$ among the remaining $n-i$ participants. 
This occurs in $r_m(n-i,k)$ ways.

\textit{Case~2:}  
If $i \geqslant m$, then there are exactly $k-1$ ties of size at least $m$ among the remaining $n-i$ participants. 
This occurs in $r_m(n-i,k-1)$ ways.

Combining these two cases, we obtain the required recurrence relation.
\end{proof} 

\begin{example} \label{mwt}
(1) The number of possible outcomes for a 5-participant race with no ties of size at least $3$ can be determined by considering three cases: 
(i) two tied groups of size $2$: $\frac{1}{2!} \binom{5}{2}\binom{3}{2} (3!)=90$ outcomes;
(ii) one tied group of size $2$: $\binom{5}{2}(4!)=240$ outcomes;
(iii) no tied group of size $2$: $5!=120$ outcomes. 
The total is $90+240+120=450$ outcomes.

(2) The number of possible outcomes for a 5-participant race with exactly one tie of size at least $3$ can be determined by considering four cases: 
(i) one tied group of size $3$ and one tied group of size $2$: $\binom{5}{3}(2!)=20$ outcomes;
(ii) exactly one tied group of size $3$ and two groups of size $1$: $\binom{5}{3}(3!)=60$ outcomes;
(iii) exactly one tied group of size $4$: $\binom{5}{4}(2!)=10$ outcomes;
(iv) exactly one tied group of size $5$: 1 outcome. 
The total is $20+60+10+1=91$ outcomes.
\end{example}

\begin{example} 
Table~\ref{table1} collects certain values of $r_2(n,k)$. 
For example, when $n=5$, the recurrence relation \eqref{recur1} can be written as
\begin{equation*}
 r_2(5,k)
=\binom{5}{1} r_2(4,k) + \sum_{i=2}^5 \binom{5}{i} r_2(5-i,k-1)
\quad\mbox{for $k=0,1,2$}.
\end{equation*}
We find 
\begin{align*}
  r_2(5,0)
&=5!=120, \\
  r_2(5,1)
&=\binom{5}{1} r_2(4,1) + \sum_{i=2}^5 \binom{5}{i} r_2(5-i,0) \\
&=5 \cdot 45 + 10 \cdot 6 + 10 \cdot 2 + 5 \cdot 1 + 1 \cdot 1
 =311, \\      
  r_2(5,2)
&=\binom{5}{1} r_2(4,2) + + \sum_{i=2}^5 \binom{5}{i} r_2(5-i,2) \\
&=5 \cdot 6 + 10 \cdot 7 + 10 \cdot 1 + 5 \cdot 0 + 1 \cdot 0
 =110.
\end{align*}

The answer to Problem~17 from the 2024 AMC 10B is $r_2(5,1)=311$ outcomes.

\begin{table}[h!]
\centering
\caption{Certain values of $r_2(n,k)$, with unlisted entries being zero.}
\label{table1}
\begin{tabular}{|c||r|r|r|r|r||r|}
\hline
\diagbox{$n$}{$k$} & $0$ & $1$ & $2$ & $3$ & $4$ & $J_n$ \\
\hline \hline
$0$ & $1$ &  &  &  &  \hphantom{143640} & 1 \\
\hline
$1$ & $1$ &  &  &  &  & $1$ \\
\hline
$2$ & $2$ & $1$ & & &  & $3$ \\
\hline
$3$ & $6$ & $7$ & & &  & $13$ \\
\hline
$4$ & $24$ & $45$ & $6$ & & & $75$ \\
\hline
$5$ & $120$ & $311$ & $110$ & & & $541$ \\
\hline
$6$ & $720$ & $2383$ & $1490$ & $90$ & & $4683$ \\
\hline
$7$ & $5040$ & $20301$ & $18802$ & $3150$ & & $47293$ \\
\hline
$8$ & $40320$ & $191369$ & $238126$ & $73500$ & $2520$ & $545835$ \\
\hline
$9$ & $362880$ & $1982971$ & $3119622$ & $1478148$ & $143640$ & $7087261$ \\
\hline
\end{tabular}
\end{table}
\end{example}

\begin{remark}
As shown in Table \ref{table1}, the row totals perfectly align with the sequence of ordered Bell numbers.
This relationship holds because for $m \in \mathbb{N}$ and $n \in \mathbb{Z}_{\geqslant 0}$,
\begin{equation*}
 J_n
=\sum_{k=0}^\infty r_m(n,k)
=\sum_{k=0}^{\left\lfloor n/m \right\rfloor} r_m(n,k).
\end{equation*}
\end{remark}

\section{Exponential Generating Functions} \label{egf} 
For $m \in \mathbb{N}$ and $k \in \mathbb{Z}_{\geqslant 0}$, the exponential generating function associated with the sequence $\{ r_m(n,k) \}_{n=0}^\infty$ is defined as
\begin{equation*}
g_{m,k}(x)=\sum_{n=0}^\infty r_m(n,k) \frac{x^n}{n!}.
\end{equation*}

\begin{theorem} \label{generating-function}
Let $T_{m-1}(x)$ be the $(m-1)$--th order Taylor polynomial of $e^x$ near $x=0$, that is,
$T_{m-1}(x)=\sum_{n=0}^{m-1} \frac{x^n}{n!}$.
Then
\begin{equation} \label{gfm} 
g_{k,m}(x)=\frac{[e^x-T_{m-1}(x)]^k}{[2-T_{m-1}(x)]^{k+1}}.
\end{equation}
\end{theorem} 

\begin{proof} 
(1) If $k=0$, then \eqref{recur1} becomes 
$r_m(n,0)=\sum_{i=1}^{m-1} \binom{n}{i} r_m(n-i,0)$ and
\begin{align*}
  g_{0,m}(x) 
&=\sum_{n=0}^{\infty} r_m(n,0) \frac{x^n}{n!} \\
&=1+\sum_{n=1}^{\infty} \left[ \sum_{i=1}^{m-1} \binom{n}{i} r_m(n-i,0) \right] \frac{x^n}{n!} \\
&=1+\sum_{n=1}^{\infty} \sum_{i=1}^{m-1} r_m(n-i,0) \frac{x^i}{i!} \frac{x^{n-i}}{(n-i)!} \\
&=1+\sum_{i=1}^{m-1} \frac{x^i}{i!} \sum_{n=1}^{\infty} r_m(n-i,0) \frac{x^{n-i}}{(n-i)!} \\
&=1+\sum_{i=1}^{m-1} \frac{x^i}{i!} \sum_{n=i}^{\infty} r_m(n-i,0) \frac{x^{n-i}}{(n-i)!} \\
&=1+\sum_{i=1}^{m-1} \frac{x^i}{i!} \sum_{n=0}^{\infty} r_m(n,0) \frac{x^{n}}{n!} \\
&=1+g_{0,m}(x)\sum_{i=1}^{m-1} \frac{x^i}{i!} \\
&=1+g_{0,m}(x)[T_{m-1}(x)-1].
\end{align*}
Solving for $g_{0,m}(x)$, we obtain
\begin{equation}\label{gfm0}
g_{0,m}(x)=\frac{1}{2-T_{m-1}(x)}.
\end{equation}

(2) If $k\geqslant 1$, then applying \eqref{recur1}, followed by interchanging the order of summation and re-indexing, leads to
\begin{align*}
  g_{k,m}(x) 
&=\sum_{n=1}^{\infty} r_m(n,k) \frac{x^n}{n!} \\
&=\sum_{n=1}^{\infty} \left[\sum_{i=1}^{m-1} \binom{n}{i} r_m(n-i,k)
 +\sum_{i=m}^n \binom{n}{i} r_m(n-i,k-1)\right] \frac{x^n}{n!} \\
&=\sum_{n=1}^{\infty} \sum_{i=1}^{m-1} r_m(n-i,k) \frac{x^i}{i!} \frac{x^{n-i}}{(n-i)!}
 +\sum_{n=1}^{\infty} \sum_{i=m}^n r_m(n-i,k-1) \frac{x^i}{i!} \frac{x^{n-i}}{(n-i)!} \\
&=\sum_{n=1}^{\infty} \sum_{i=1}^{m-1} r_m(n-i,k) \frac{x^i}{i!} \frac{x^{n-i}}{(n-i)!}
 +\sum_{n=m}^{\infty} \sum_{i=m}^n r_m(n-i,k-1) \frac{x^i}{i!} \frac{x^{n-i}}{(n-i)!} \\
&=\sum_{i=1}^{m-1} \frac{x^i}{i!} \sum_{n=i}^{\infty} r_m(n-i,k) \frac{x^{n-i}}{(n-i)!}
 +\sum_{i=m}^{\infty} \frac{x^i}{i!} \sum_{n=i}^{\infty} r_m(n-i,k-1) \frac{x^{n-i}}{(n-i)!} \\
&=\sum_{i=1}^{m-1} \frac{x^i}{i!} \sum_{n=0}^{\infty} r_m(n,k) \frac{x^{n}}{n!}
 +\sum_{i=m}^{\infty} \frac{x^i}{i!} \sum_{n=0}^{\infty} r_m(n,k-1) \frac{x^{n}}{n!} \\
&=g_{k,m}(x) \sum_{i=1}^{m-1} \frac{x^i}{i!}
 +g_{k-1,m}(x) \sum_{i=m}^{\infty} \frac{x^i}{i!} \\
&=g_{k,m}(x)[T_{m-1}(x)-1]+g_{k-1,m}(x)[e^x-T_{m-1}(x)].  
\end{align*}
Solving for $g_{k,m}(x)$, we obtain
\begin{equation} \label{gfm1} 
g_{k,m}(x)=\left[\frac{e^x-T_{m-1}(x)}{2-T_{m-1}(x)}\right] g_{k-1,m}(x).
\end{equation}
Using \eqref{gfm0}, \eqref{gfm1}, and mathematical induction on $k$, \eqref{gfm} follows. 
\end{proof}

\begin{remark}
(1) Since $J_n=\sum_{k=0}^\infty r_m(n,k)$, summing the exponential generating functions for the sequences $\{r_m(n,k)\}_{n=0}^{\infty}$ yields
\begin{equation*}
 \sum_{k=0}^\infty g_{m,k}(x)
=\frac{1}{2-T_{m-1}(x)} \sum_{k=0}^\infty \left[\frac{e^x-T_{m-1}(x)}{2-T_{m-1}(s)}\right]^k \\
=\frac{1}{2-e^x},
\end{equation*}
which coincides with the exponential generating function of $\{J_n\}_{n=0}^\infty$ (see \cite{EM82}).

(2) Note that $\lim_{m\to\infty} T_{m-1}(x)=e^x$.
If $k=0$, then from \eqref{gfm}, we have 
\begin{equation*}
\lim_{m\to \infty}g_{0,m}(x)=\frac{1}{2-e^x},
\end{equation*}
which is precisely the exponential generating function of $\{J_n\}_{n=0}^\infty$ (see \cite{EM82}).
\end{remark}

\begin{example} 
Consider the case where $m=3$. 
It corresponds to ties of size at least 3. 
In this case, we have $T_2(x)=1+x+\frac{x^2}{2}$ and 
\begin{equation*}
g_{k,3}(x)=\frac{\left(e^x-1-x-\frac{x^2}{2}\right)^k}{\left(1-x-\frac{x^2}{2}\right)^{k+1}}.
\end{equation*}
\end{example}

(1) If $k=0$, then $g_{0,3}(x)$ has the power series representation near $x=0$ as follows.
\begin{equation} \label{gf1} 
\begin{split}
  g_{0,3}(x)
&=\frac{1}{1-x-\frac{x^2}{2}} \\
&=1+x+3\cdot\frac{x^2}{2!}+12\cdot\frac{x^3}{3!}+66\cdot\frac{x^4}{4!}+450\cdot\frac{x^5}{5!}+ O(x^6)
\end{split}
\end{equation}
The coefficient of the $\frac{x^5}{5!}$-term in \eqref{gf1} gives us $r_3(5,0)=450$, which indicates that there are $450$ possible outcomes for a 5-participant race featuring no tie of size at least $3$ and aligns with the calculation presented in part~(1) of Example~\ref{mwt}.

(2) If $k=1$, then $g_{1,3}(x)$ has the power series representation near $x=0$ as follows.
\begin{equation} \label{gf2} 
\begin{split}
  g_{1,3}(x)
&=\frac{e^x-1-x-\frac{x^2}{2}}{\left(1-x-\frac{x^2}{2}\right)^2} \\
&=\frac{x^3}{3!}+9\cdot\frac{x^4}{4!}+91\cdot\frac{x^5}{5!}+ 973\cdot\frac{x^6}{6!}+ O(x^7) 
\end{split}
\end{equation}
The coefficient of the $\frac{x^5}{5!}$-term in \eqref{gf1} gives us $r_3(5,1)=91$, which indicates that there are $91$ possible outcomes for a 5-participant race featuring exactly one tie of size at least $3$ and aligns with the calculation presented in part~(2) of Example~\ref{mwt}.

\section{A Closed-Form Formula for $r_m(n,k)$} \label{Formula} 

With the exponential generating functions $g_{m,k}(x)$ given by \eqref{gfm}, we can obtain a closed-form formula for $r_m(n,k)$.

\begin{theorem} \label{formula1} 
Define
\begin{equation} \label{a_n}
a_0=1
\quad\mbox{and}\quad
 a_n
=\sum_{j=1}^n \sum_{\substack{i_1+i_2+\cdots+i_j=n \\ 1 \leqslant i_1,i_2,\ldots,i_j \leqslant m-1}} 
 \frac{n!}{i_1!i_2! \cdots i_j!}
\quad\mbox{for $n \in \mathbb{N}$}.
\end{equation}
For $k, m \in \mathbb{N}$ and $n \in \mathbb{Z}_{\geqslant 0}$, we have $r_m(n,0)=a_n$, and
\begin{equation}\label{form01}
 r_m(n,k)
=\sum_{l=km}^n \binom{n}{l} a_{n-l} 
  \sum_{\substack{l_1+l_2+\cdots+l_k=l \\ l_1,l_2,\ldots,l_k \geqslant m}}
  \frac{l!}{l_1!l_2! \cdots l_k!} \prod_{i=1}^k \left[\sum_{j=m}^{l_i}\binom{l_i}{j}a_{l_i-j}\right].
\end{equation}
\end{theorem}

\begin{proof}
Define 
\begin{equation*}
p_{m-1}(x)=\sum_{i=1}^{m-1} \frac{x^i}{i!}
\quad\mbox{and}\quad
q_{m}(x)=\sum_{i=m}^\infty \frac{x^i}{i!}.
\end{equation*}
Then $T_{m-1}(x)=1+p_{m-1}(x)$, $e^x=T_{m-1}(x)+q_m(x)$, and hence
\begin{equation*}
g_{k,m}(x)=\left[ \frac{q_m(x)}{1-p_{m-1}(x)} \right]^k \frac{1}{1-p_{m-1}(x)}.
\end{equation*}
First, rewrite the rational function $\frac{1}{1-p_{m-1}(x)}$ as a power series near $x=0$. 
\begin{align*}
  \frac{1}{1-p_{m-1}(x)}
&=\sum_{n=0}^\infty [p_{m-1}(x)]^n \\
&=1+\sum_{n=1}^\infty \sum_{j=n}^\infty 
  \sum_{\substack{i_1+i_2+\cdots+i_n=j \\ 1 \leqslant i_1,i_2,\ldots,i_n \leqslant m-1}} 
  \frac{x^{i_1}}{i_1!} \frac{x^{i_2}}{i_2!} \cdots \frac{x^{i_n}}{i_n!}  \\
&=1+\sum_{n=1}^\infty \sum_{j=n}^\infty 
  \sum_{\substack{i_1+i_2+\cdots+i_n=j \\ 1 \leqslant i_1,i_2,\ldots,i_n \leqslant m-1}} 
  \frac{x^j}{i_1! i_2! \cdots i_n!}  \\
&=1+\sum_{j=1}^\infty \sum_{n=1}^j 
  \sum_{\substack{i_1+i_2+\cdots+i_n=j \\ 1 \leqslant i_1,i_2,\ldots,i_n \leqslant m-1}}
  \frac{x^j}{i_1!i_2! \cdots i_n!} \\
&=1+\sum_{n=1}^\infty \sum_{j=1}^n 
  \sum_{\substack{i_1+i_2+\cdots+i_j=n \\ 1 \leqslant i_1,i_2,\ldots,i_j \leqslant m-1}}
  \frac{n!}{i_1!i_2! \cdots i_j!} \frac{x^n}{n!} \\
&=\sum_{n=0}^\infty a_n \frac{x^n}{n!}.
\end{align*}
We obtain $r_m(n,0)=a_n$ for $n \in \mathbb{Z}_{\geqslant 0}$. 
\begin{equation*}
 \frac{q_m(x)}{1-p_{m-1}(x)}
=\sum_{n=m}^\infty \sum_{l=m}^n \frac{x^l}{l!} \left[a_{n-l} \frac{x^{n-l}}{(n-l)!}\right] \\
=\sum_{n=m}^\infty \sum_{l=m}^n \binom{n}{l} a_{n-l} \frac{x^n}{n!}
\end{equation*}
Define 
\begin{equation*}
\bar{a}_n=\sum_{i=m}^n \binom{n}{i} a_{n-i}
\quad\mbox{for $n \in \mathbb{N}$}.
\end{equation*}
Then 
\begin{equation*} 
\frac{q_m(x)}{1-p_{m-1}(x)}=\sum_{n=m}^\infty \bar{a}_n \frac{x^n}{n!}
\end{equation*}
and 
\begin{align*}
  \left[ \frac{q_m(x)}{1-p_{m-1}(x)} \right]^k
&=\sum_{n=km}^\infty \sum_{\substack{l_1+l_2+\cdots+l_k=n \\ l_1,l_2,\ldots,l_k \geqslant m}}
  \left(\bar{a}_{l_1}\frac{x^{l_1}}{l_1!}\right) \left(\bar{a}_{l_2}\frac{x^{l_2}}{l_2!}\right)
  \cdots \left(\bar{a}_{l_k}\frac{x^{l_k}}{l_k!}\right) \\
&=\sum_{n=km}^\infty \sum_{\substack{l_1+l_2+\cdots+l_k=n \\ l_1,l_2,\ldots,l_k \geqslant m}}
  \frac{n!}{l_1!l_2! \cdots l_k!} \bar{a}_{l_1}\bar{a}_{l_2} \cdots \bar{a}_{l_k} \frac{x^n}{n!}. 
\end{align*}
Define 
\begin{equation*}
b_n=\sum_{\substack{l_1+l_2+\cdots+l_k=n \\ l_1,l_2,\ldots,l_k \geqslant m}}
  \frac{n!}{l_1!l_2! \cdots l_k!}\bar{a}_{l_1}\bar{a}_{l_2} \cdots \bar{a}_{l_k} 
\quad\mbox{for $n \in \mathbb{N}$}.
\end{equation*}
Then 
\begin{equation*}
 \left[ \frac{q_m(x)}{1-p_{m-1}(x)} \right]^k
=\sum_{n=km}^\infty b_n\frac{x^n}{n!}.
\end{equation*}
Thus,
\begin{align*}
  g_{k,m}(x)
&=\sum_{n=km}^\infty \sum_{l=km}^n \left(b_l\frac{x^l}{l!}\right)
  \left[a_{n-l}\frac{x^{n-l}}{(n-l)!}\right] \\
&=\sum_{n=km}^\infty \sum_{l=km}^n \binom{n}{l} a_{n-l}b_l \frac{x^n}{n!}.
\end{align*}
We obtain
\begin{align*} 
  r_m(n,k)
&=\sum_{l=km}^n \binom{n}{l} a_{n-l}b_l \\
&=\sum_{l=km}^n \binom{n}{l} a_{n-l} 
  \sum_{\substack{l_1+l_2+\cdots+l_k=l \\ l_1,l_2,\ldots,l_k \geqslant m}}
  \frac{l!}{l_1!l_2! \cdots l_k!} \prod_{i=1}^k \left[\sum_{j=m}^{l_i}\binom{l_i}{j}a_{l_i-j}\right].
\qedhere
\end{align*}
\end{proof}

\section{Two Special Cases} \label{Special-cases}
We consider two interesting special cases in Problem~\ref{prob:snail-race} as follows.

\textit{Case~1:} $m=1$

By \eqref{a_n}, we have $a_0=1$ and $a_n=0$ for $n \in \mathbb{N}$.
Then $r_m(0,0)=a_0=1$ and $r_m(n,0)=a_n=0$ for $n \in \mathbb{N}$. 
For $n \in \mathbb{Z}_{\geqslant 0}$ and $k \in \mathbb{N}$, it follows from \eqref{form01} that
\begin{equation} \label{special-case1}
 r_1(n,k)
=\sum_{\substack{l_1+l_2+\cdots+l_k=n \\ l_1,l_2,\ldots,l_k \geqslant 1}} \frac{n!}{l_1!l_2! \cdots l_k!}
=k! S(n,k)
=\sum_{i=1}^n (-1)^{k-i} \binom{k}{i} i^n,
\end{equation}
where the $S(n,k)$ are Stirling numbers of the second kind. 
This matches the closed-form formula for $J_{nk}$ in \cite{EM82}.

\textit{Case~2:} $m=2$

By \eqref{a_n}, we have $a_n=n!$ for $n \in \mathbb{Z}_{\geqslant 0}$.
Then $r_2(n,0)=a_n=n!$ for $n \in \mathbb{Z}_{\geqslant 0}$.
For $n \in \mathbb{Z}_{\geqslant 0}$ and $k \in \mathbb{N}$, it follows from \eqref{form01} that 
\begin{equation*}
 r_2(n,k)
=\sum_{l=2k}^n \frac{n!}{l!} \sum_{\substack{l_1+l_2+\cdots+l_k=l \\ l_1,l_2,\ldots,l_k \geqslant 2}} 
  \frac{l!}{l_1! l_2! \cdots l_k!} \prod_{i=1}^k \left(\sum_{j=2}^{l_i} \frac{l_i!}{j!}\right). 
\end{equation*}
Define 
\begin{equation} \label{Delta_nk}
 \Delta_n^k
=\{(l_1,l_2, \ldots, l_k) \mid l_1+l_2+\cdots+l_k=n 
 \ \mbox{ and } \ l_1,l_2,\ldots,l_k \geqslant 2\}.
\end{equation}
and
\begin{equation} \label{c_n}
c_n=\sum_{i=2}^n \frac{n!}{i!}
\quad\mbox{for $n \in \mathbb{N}$}. 
\end{equation}
Then \eqref{form01} can be expressible as
\begin{equation} \label{special-case2}
 r_2(n,k)
=\sum_{l=2k}^n \frac{n!}{l!} \sum_{(l_1,l_2,\ldots,l_k) \in \Delta_l^k} 
 \frac{l!}{l_1! l_2! \cdots l_k!} \prod_{i=1}^k c_{l_i},
\end{equation}

\begin{example} \label{example6} 
By \eqref{special-case2}, we obtain
\begin{equation} \label{rn1}
  r_2(n,1)
=\sum_{l=2}^n \frac{n!}{l!} \sum_{i=2}^l \frac{l!}{i!} 
=\sum_{l=2}^n \sum_{i=2}^l \frac{n!}{i!}
=\sum_{i=2}^n \sum_{l=i}^n \frac{n!}{i!}
=\sum_{i=2}^n \frac{n!(n+1-i)}{i!}.
\end{equation}
\end{example}

\begin{example}
Find $r_2(6,1)$, $r_2(6,2)$ and $r_2(6,3)$.

Using \eqref{rn1}, we obtain
\begin{equation*}
 r_2(6,1)
=\sum_{i=2}^6 \frac{6!(7-i)}{i!}
=2383.
\end{equation*}

By \eqref{c_n}, we have $c_2=1$, $c_3=4$, and $c_4=17$.
By \eqref{Delta_nk}, we also have 
\begin{equation*}
\Delta_n^2=\{(2,n-2),(4,n-4),\ldots,(n-2,2)\}
\quad \mbox{and} \quad
\Delta_6^3=\{(2,2,2)\}.
\end{equation*} 
Then 
\begin{equation} \label{rn2}
 r_2(6,2)
=\sum_{l=4}^6 \frac{6!}{l!} \sum_{i=2}^{l-2} \frac{l!}{i!(l-i)!} c_ic_{l-i}
=1490
\end{equation}
and 
\begin{equation*}
 r_2(6,3)
=\frac{6!}{6!}  \frac{6!}{2!2!2!} c_2^3
=90.
\end{equation*}
\end{example}

\section{A Second Closed-Form Formula for $r_2(n,k)$} \label{BellP} 

Alternatively, another closed-form formula for $r_2(n,k)$ can be obtained by evaluating the $n$-th derivative of the exponential generating function $g_{2,k}(x)$ at $x=0$. 
We first recall the definition of the partial Bell polynomials and Fa\`{a} di Bruno's Formula (see Bell \cite{BellPoly}, Comtet \cite{Comtet74}, Johnson \cite{Johnson02}, Riordan \cite{Riordan80}).

\begin{definition}[Partial Bell Polynomials]\label{def-BellP}
The \emph{partial Bell polynomials} are given by
\begin{equation}\label{bellp01}
   B_{i,k}(x_1, x_2, \cdots, x_{i-k+1}) 
   = \sum \frac{i!}{s_1!s_2!\cdots s_{i-k+1}!}\prod_{j=1}^{i-k+1}
\left(\frac{x_j}{j!}\right)^{s_j},
\end{equation}
where the summation takes place over all $s_1,s_2,\ldots,s_{i-k+1} \in \mathbb{Z}_{\geqslant 0}$ such that 
\begin{equation}\label{bellp02}
\left\{
\begin{array}{rcl}
  s_1+s_2+\cdots+s_{i-k+1} & = &  k, \\ 
  s_1+2s_2+\cdots+(i-k+1)s_{i-k+1} & = & i.
\end{array}
\right.
\end{equation}
\end{definition}

\begin{theorem}[Fa\`{a} di Bruno's Formula]\label{thm-faa}
If $f$ and $\phi$ are functions with a sufficient number of derivatives, then
\begin{equation}\label{faa}
 \frac{d^i}{dx^i}f(\phi(x))
=\sum \frac{i!}{s_1!s_2!\cdots s_i!}f^{(s_1+s_2+\cdots+s_i)}(\phi(x))
 \prod_{j=1}^i \left[\frac{\phi^{(j)}(x)}{j!}\right]^{s_j},
\end{equation}
where the summation is over all $i$-tuples $(s_1, s_2, \ldots, s_i)$ of nonnegative integers satisfying the constraint $s_1+2s_2+\cdots+is_i=i$.
\end{theorem}

The next theorem provides another closed-form formula for $r_2(n,k)$.

\begin{theorem}\label{thm-formula2} 
For $n,k \in \mathbb{N}$,  
\begin{equation}\label{bellp05} 
r_2(n,k)=\sum_{i=2}^n \binom{n}{i}(k+n-i)! B_{i,k}(0,1,1,\ldots,1),
\end{equation}
where the $B_{i,k}$ are the partial Bell polynomials defined in \eqref{bellp01}.
\end{theorem}

\begin{proof} 
We begin by applying the Leibniz formula for the $n$-th derivative of the exponential generating function $g_{2,k}(x)$:
\begin{equation}\label{leibniz} 
\begin{split}
  \left[g_{2,k}(x)\right]^{(n)}  
&=\left[(e^x-1-x)^k (1-x)^{-(k+1)}\right]^{(n)} \\
&=\sum_{i=0}^n \binom{n}{i} \left[(e^x-1-x)^k\right]^{(i)}\left[(1-x)^{-(k+1)}\right]^{(n-i)} \\
&=\sum_{i=0}^n {\binom{n}{i}}\frac{(k+n-i)!}{k!} \left[(e^x-1-x)^k\right]^{(i)} (1-x)^{-(k+1+n-i)}.
\end{split}
\end{equation}

Next, we evaluate the $i$-th order derivative of $(e^x-1-x)^k$ at $x=0$ using Fa\`{a} di Bruno's formula with $f(x)=x^k$ and $\phi(x)=e^x-1-x$. 
Note that $\phi^{\prime}(x)=e^x-1$ and $\phi^{(j)}(x)=e^x$ for $j \geqslant  2$. 
Evaluating these at $x=0$, we find $\phi(0)=\phi^{\prime}(0)=0$ and $\phi^{(j)}(0)=1$ for $j \geqslant 2$. 

In Fa\`{a} di Bruno's formula \eqref{faa}, a term is nonzero at $x=0$ only if $s_1=0$ (for $\phi^{\prime}(0)=0$) and $s_1+s_2+\cdots+s_i=k$ (for $f^{(j)}(0)=0$ for $j \ne k$ and $f^{(k)}(0)=k!$). 
Thus, the summation in \eqref{faa} is constrained by $s_1=0$ and $\sum_{j=1}^i s_j=k$, along with the original constraint $\sum_{j=1}^i js_j=k$. 
These constraints simplify to the system
\begin{equation} \label{bellp03}
\left\{
\begin{array}{rcl}
 s_2+s_3\cdots+s_{i-k+1} &=& k\\
 2s_2+3s_3+\cdots+(i-k+1)s_{i-k+1}&=&i\\
\end{array}
\right.
\end{equation}
and hence
\begin{equation} \label{bellp04}
\begin{split}
  \left[(e^x-1-x)^k\right]^{(i)}\Bigr\vert_{x=0} 
&=k!\sum \frac{i!}{s_2!s_3!\cdots s_i!}\prod_{j=2}^i \left(\frac{1}{j!}\right)^{s_j} \\
&=k! \vphantom{\frac{1}{2}}B_{i,k}(0,1,1,\ldots,1),
\end{split}
\end{equation}
where the $B_{i,k}$ are the partial Bell polynomials defined in \eqref{bellp01}.  

Finally, setting $x=0$ in \eqref{leibniz} and using \eqref{bellp04}, we obtain the desired result
\begin{equation}\label{bellp05a}
 r_2(n,k)
=\left[g_k(x)\right]^{(n)}\Bigr\vert_{x=0}= \sum_{i=2}^n \binom{n}{i}(k+n-i)! B_{i,k}(0,1,1,\ldots,1).
\end{equation}
The lower limit of the summation is $i=2$ because when $i=0$ or $i=1$, the derivative of $(e^x-1-x)^k$ at $x=0$ is zero, as indicated by the conditions on $s_1$ and $k$. 
\end{proof}

\begin{remark} 
The value of a partial Bell polynomial is zero if its defining summation conditions, given by \eqref{bellp02} or \eqref{bellp03}, has no nonnegative integer solution.  
Consequently, if $i<2k$,  $B_{i,k}(0,1,\ldots,1)=0$ because then \eqref{bellp03} has no nonnegative integer solution.
\end{remark}

\begin{example} (1) 
Consider the case where $k=1$ and $n \geqslant 2$. From \eqref{bellp05},
\begin{equation*}
r_2(n,1)=\sum_{i=2}^n \binom{n}{i}(n+1-i)! B_{i,1}(0,1,1, \ldots,1). 
\end{equation*}
For $2 \leqslant i \leqslant n$, \eqref{bellp03} becomes
\begin{equation*}
\left\{\begin{array}{rcl}
  s_2+s_3+\cdots+s_{i} &= & 1\\
  2s_2+3s_3+\cdots+is_{i}&=&i \\
\end{array}
\right.
\end{equation*}
which has the unique solution $(s_2, s_3, \ldots, s_i)=(0,0, \ldots, 1)$. 
So $B_{i,1}(0,1,1,\ldots,1)$ simplifies to $1$. 
Therefore, 
\begin{equation*}
 r_2(n,1)
=\sum_{i=2}^n \binom{n}{i}(n+1-i)! 
=\sum_{i=2}^n \frac{n!(n+1-i)}{i!},
\end{equation*}
which matches with \eqref{rn1} in Example~\ref{example6}.\\

(2) Consider the case where $k=2$ and $n=6$. 
Since $B_{2,2}(0,1)=B_{3,2}(0,1)=0$,
\begin{equation*}
\begin{array}{lcl}
   B_{4,2}(0,1,1) & = &\displaystyle \frac{4!}{2!}\left(\frac{1}{2!}\right)^2=3, 
   \\ 
   B_{5,2}(0,1,1,1)& = & \displaystyle \frac{5!}{1!1!}\left(\frac{1}{2!}\right)^1\left(\frac{1}{3!}\right)^1=10, \\
   B_{6,2}(0,1,1,1,1) & = & \displaystyle \frac{6!}{1!1!}\left(\frac{1}{2!}\right)^1
           \left(\frac{1}{4!}\right)^1
     +\frac{6!}{2!}\left(\frac{1}{3!}\right)^2=25
\end{array}
\end{equation*}
from \eqref{bellp05}, we have
\begin{align*}
  r_2(6,2)  
&=\sum_{i=2}^6 \binom{6}{i}(8-i)! B_{i,2}(0,1, \cdots,1)
 =1490,
\end{align*}
which matches with \eqref{rn2} in Example \ref{example6}.
\end{example}

\textbf{Disclosure Statement.}
The authors have no conflicts of interest to disclose. 

\bibliographystyle{vancouver}
\bibliography{Snail_Race_Bib}

@Article{	  BellNum,
  author   = {Bell, E. T},
  title		= {Exponential Numbers.},
  journal	= {Amer. Math. Monthly},
  volume	= {41},
  pages		= {411--419},
  year		= 1934
}

@Article{	  BellPoly,
  author   = {Bell, E. T},
  title		= {Exponential Polynomials.},
  journal	= {Ann. Math.},
  volume	= {2},
  number	= {2},
  pages		= {258--277},
  year		= 1934
}

@Article{	  Cayley1859,
  author     = {Cayley, A},
  title		= {On the analytic forms called trees, second part.},
  journal	= {Philos. Mag.},
  volume = {18},
  pages		= {374--378},
  year		= 1859
}

@Book{		Comtet74,
  author	= {Comtet, L},
  title		= {Advanced Combinatorics: The Art of Finite and Infinite Expansions},
  publisher	= {Reidel Publishing Company},
  address	= {Dordrecht, Holland / Boston, U.S.A.},
  year		= 1974
}

@Article{  EM82,
  author   ={Mendelson, E},
  title	= {Races with Ties.},
  journal	 = {Math. Mag.},
  volume 	= {55},
  number	= {3},
  pages		= {170--175},
  year		= 1982
}

@Book{		GKP94,
  author	= {Graham, R and Knuth, D and Patashnik, O},
  title		= {Concrete Mathematics: A Foundation for Computer Science},
  publisher	= {Addison-Wesley},
  address	= {Reading, MA},
  edition	= {2},
  pages		= {257--267},
  year		= 1994
}

@Article{  Good75,
  author	= {Good, I},
  title		= {The number of orderings of $n$ candidates when ties are permitted.},
  journal	= {Fibonacci Quartely},
  volume 	= {13},
  pages		= {11--18},
  year		= 1975
}

@Article{  Gross62,
  author   ={Gross, O},
  title	= {Preferential Arrangements.},
  journal	 = {Amer. Math. Monthly},
  volume 	= {69},
  number	= {1},
  pages		= {4--8},
  year		= 1962
}

@Article{	Johnson02,
  author   ={Johnson, W},
  title	= {The Curious History of Fa\'{a} di Bruno's Formula.},
  journal	 = {Amer. Math. Monthly},
  volume 	= {109},
  number	= {3},
  pages		= {217--234},
  year		= 2002
}

@Book{		Riordan80,
  author	= {Riordan, J},
  title		= {An Introduction to Combinatorial Analysis},
  publisher	= {Princeton University Press},
  address	= {Princeton, NJ},
  year		= 1980
}

@Article{  VC95,
  author   ={Velleman, D and Call, G},
  title	= {Permutations and Combination Locks.},
  journal	 = {Math. Mag.},
  volume 	= {68},
  number	= {4},
  pages		= {243--253},
  year		= 1995
}

\end{document}